\documentclass[11pt,oneside]{article}

\usepackage{latexsym}
\usepackage{shortvrb}
\usepackage{graphicx}
\usepackage{fancyhdr}
\usepackage{amsfonts}
\usepackage{amsmath}
\usepackage{amssymb}
\usepackage{mathrsfs}
\usepackage{makeidx}
\usepackage[all]{xy}
\usepackage{tikz}
\usepackage{amsthm}

\usepackage[T1]{fontenc}
\usepackage{amsfonts}
\usepackage{amsmath}

\newtheorem{thm}{Theorem}[section]
\newtheorem{lem}[thm]{Lemma}

\theoremstyle{definition}

\newtheorem{prob}[thm]{Problem}
\newtheorem{conj}[thm]{Conjecture}

\newcommand{\blackged}{\hfill$\blacksquare$}
\newcommand{\whiteged}{\hfill$\square$}
\newcounter{proofcount}

\renewenvironment{proof}[1][\proofname.]{\par
  \ifnum \theproofcount>0 \pushQED{\whiteged} \else \pushQED{\blackged} \fi%
  \refstepcounter{proofcount}
  \normalfont %\topsep6\p@\@plus6\p@\relax
  \trivlist
  \item[\hskip\labelsep
        \itshape
    {\bf\em #1}]\ignorespaces
}{%
  \addtocounter{proofcount}{-1}
  \popQED\endtrivlist
}

\begin{document}

\begin{center}
\textbf{\large{A Note on the Existence of Indecomposable Essential Submodules of Direct Sums of Copies of the Ring of Quotients of Ore Domains}}
\end{center}

\begin{center}
Juan Orendain
\end{center}

\noindent We study the problem of existence of essential indecomposable submodules of direct sums of copies of the ring of quotients of Ore domains. We provide, for each Ore domain $D$, with at least three non-associate central irreducibles, and with no non-zero infinitely divisible elements, a lower bound for the supremum of all cardinals $\kappa$ such that the direct sum of $\kappa$ copies of the ring of quotients of $D$ contains indecomposable essential $D$-submodules, and a lower bound for the number of times this bound is attained up to isomorphisms. We provide explicit examples illustrating these results. 

\section{Introduction}

\noindent Solutions to the problem of existence of essential indecomposable submodules of direct sums of copies of the ring of quotients of the integers, that is, solutions to the problem of existence of indecomposable torsion-free abelian groups of given ranks, can be traced back to \cite{F4} or even \cite{K}. We study, using essentially ideas appearing in \cite{F4}, the more general problem of existence and uniqueness of essential indecomposable submodules of direct sums of copies of the ring of quotients of Ore domains. For a given Ore domain $D$, admitting sets of at least three non-associate central irreducibles, and not admitting non-zero infinitely divisible elements, we provide a lower bound for the supremum of all cardinals $\kappa$ such that the direct sum of $\kappa$ copies of the ring of quotients of $D$ admits essential indecomposable $D$-submodules, and a lower bound for the number of times this bound is attained up to isomorphisms. We provide explicit examples of these types of calculations.

In what follows the letter $D$ will, unless otherwise stated, always stand for a left Ore domain, the letter $E$ will always denote the left ring of quotients of $D$, the word \textsl{module} will mean a left unital $D$-module, and the word \textsl{morphism} will always mean a morphism in $D$-Mod. The choice of side on which the Ore condition is required on $D$ is arbitrary, and it is made for convenience, it should be evident nevertheless that all statements and proofs work for right Ore domains just as well. We will say that an element $a\in D$ is irreducible in $D$ if the equation $a=bc$ in $D$ implies that either $b$ or $c$ is a unit in $D$. Given central irreducibles $a,b\in D$, we will say that $a$ and $b$ are associates in $D$ if there exists a unit $u\in D$ such that $a=ub$. We will say that a non-zero element $a\in R$ is infinitely left divisible in $R$ if there exists $b\in R\setminus U(R)$ such that for every $n\geq 1$, there exists $x_n\in R$ such that $a=b^nx_n$. We will denote by $\mathfrak{S}$ the class of all modules $M$ such that $M$ can be embedded in a direct sum of copies of $E$. Given $M\in\mathfrak{S}$, we will write $rk(M)$ for the minimal cardinal $\kappa$ such that $M$ can be embedded in $E^{(\kappa)}$. Finally, we will denote by $\chi(D)$ the supremum of all cardinals $\kappa$ such that there exists an indecomposable module $M$ such that $rk(M)=\kappa$. $\chi(D)$ will thus take values in the class of all cardinals together with an artificial element $*$ satisfying the convention that $\kappa+*=*$ for every cardinal $\kappa$. Observe that $\chi(D)=1$ if $D$ is a division ring. More generally, in \cite{K} it is proven that if $D$ is a valuation ring, then $\chi(D)=1$ if and only if $D$ is a maximal discrete valuation ring. In particular, for every prime $p$, the ring of $p$-adic integers, $\mathbb{Z}_p$, satisfies equation 

\[\chi(\mathbb{Z}_p)=1\]

\noindent We will be interested in statements of the form:

\[\chi(D)\geq \kappa\]

\noindent \textit{and this bound is attained at least $\mu$ times up to isomorphisms}. Any such statement should be translated as: \textit{There exist at least $\mu$ non-isomorphic indecomposable modules $M$, in $\mathfrak{S}$, such that $rk(M)=\kappa$}. Or equivalently, as: \textit{There exist at least $\mu$ non-isomorphic essential indecomposable submodules of the direct sum of $\kappa$ copies of $E$}. The following will be the main result of this note

\begin{thm}
Let $\kappa_1,\kappa_2$ be cardinals. Suppose $D$ admits sets of $2\kappa_1+\kappa_2$ non-associate central irreducibles and no non-zero element of $R$ is infinitely left divisible in $R$. Then

\[\chi(D)\geq 2^{\kappa_1}\]

\noindent and this bound is attained at least $2^{\kappa_2}-1$ times up to isomorphisms.  
\end{thm}

\noindent We postpone a proof of Theorem 1.1 until section 3. In section 2 we present explicit examples of Ore domains satisfying the conditions of Theorem 1.1.
 
\section*{Acknowledgements}
\noindent The author would like to dedicate this paper to the late professor Francisco Raggi, without whose attention, guidance, and friendship this paper would have not been possible

\section{Examples}

\noindent In this section we provide explicit examples of Ore domains satisfying the conditions of Theorem 1.1. Observe first that left noetherian domains are examples of left Ore domains that do not admit non-zero infinite left divisible elements. We begin by presenting examples of commutative domains satisfying the conditions of Theorem 1.1. 

\begin{enumerate}
\item Let $a\in\mathbb{Z}$ be squarefree. Let $\Omega_1,\Omega_2$ be (possibly empty) finite sets such that $\Omega_1\cup\Omega_2$ is algebraically independent over $\mathbb{Z}$. If in this case we make $D$ to be equal to $\mathbb{Z}[\sqrt{a}][\Omega_1][[\Omega_2]]$, then $D$ is a commutative noetherian domain admitting sets of $\aleph_0$ non-associate irreducibles. Thus, by Theorem 1.1 we have

\[\chi(D)\geq 2^{\aleph_0}\] 

\noindent and this bound is attained at least $2^{\aleph_0}$ times up to isomorphisms. We consider this result as a generalization of results obtained in \cite{F3}.

\item Let $k$ be a field. Let $\Omega_1,\Omega_2$ be finite sets such that $\Omega_1\cup\Omega_2$ is algebraically independent over $k$. Suppose that $\Omega_1\neq\emptyset$. If in this case we make $D$ to be equal to $k[\Omega_1][[\Omega_2]]$, then $D$ is again a commutative noetherian domain, now admitting sets of $\aleph_0\left|k\right|$ non-associate irreducibles. Thus, by Theorem 1.1 we now have

\[\chi(D)\geq 2^{\aleph_0\left|k\right|}\]

\noindent and this bound is attained at least $2^{\aleph_0\left|k\right|}$ times up to isomorphisms.
 
\item Let $R$ be a unique factorization domain. Let $a\in R$. Suppose $a$ admits $n$ ($n\geq 3$) non-associate prime divisors in $R$. If in this case we make $D$ to be equal to the localization $R_{(a)}$, then $D$ is a commutative domain not admitting infinitely divisible elements and admitting exactly $n$ non-associate irreducibles. Thus, by Theorem 1.1 we now have

\[\chi(D)\geq 2^{\left\lfloor n/2\right\rfloor-1}\]  

\item Let $k$ be an algebraically closed field. Let $V$ be an algebraic variety over $k$. Let $a\in V$. Suppose $V$ is smooth at $a$. In this case, if we make $D$ to be equal to $\mathcal{O}_{V,a}$, then $D$ is, by the Auslander-Buchsbaum Theorem \cite{AB}, a noetherian UFD. Thus, if $V$ admits sets of $2\kappa_1+\kappa_2$ subvarieties containing $a$, $D$ admits sets of $2\kappa_1+\kappa_2$ non associate irreducibles. In particular, if $V$ is a hyperplane of dimension $\geq 2$, by Theorem 1.1 we have

\[\chi(D)\geq 2^{\left|k\right|}\]

\noindent and this bound is attained at least $2^{\left|k\right|}$ times up to isomorphisms. 

\end{enumerate}

\noindent Observe that example 1 above generalizes the results of \cite{F3}. Observe also that the bounds obtained in this example are weak compared to those obtained in \cite{F4}. We regard the following conjecture as a possible generalization of the results of \cite{F4} and a possible improvement of Theorem 1.1

\begin{conj}
There exists a model in ZFC in which, for every pair of cardinals $\kappa_1,\kappa_2$, and for every Ore domain $D$, such that $D$ admits infinite sets of non-associate central irreducibles and no non-zero element of $D$ is infinitely left divisible in $R$, we have

\[\chi(D)\geq 2^{\kappa_1}\]

\noindent and this bound is attained at least $2^{\kappa_2}-1$ times up to isomorphisms.
\end{conj}

\noindent We now present examples of non-commutative Ore domains satisfying the conditions of Theorem 1.1. Let $R$ be a commutative noetherian domain, let $\sigma$ be an automorphism of $R$, and let $\delta$ be a $\sigma$-derivation of $R$. In this case, the skew polynomial ring $R[y;\sigma,\delta]$ is a left-right noetherian domain (see \cite{R}), non-commutative if and only if $\sigma\neq id$ or $\delta\neq 0$. Observe that, in this case, every irreducible in $R$ is irreducible in $R[y;\sigma,\delta]$, as a degree 0 polynomial, and that a central element $a\in R$ is central in $R[y;\sigma,\delta]$, as a degree 0 polynomial, if and only if identities

\[\sigma(a)=a \ \mbox{and} \ \delta(a)=0\]

\noindent hold. That is, an element $a\in R$ is central in $R[y;\sigma,\delta]$, as a degree 0 polynomial, if and only if $a$ is central in $R$, $a$ is fixed by $\sigma$, and $a$ is a constant with respect to $\delta$. Moreover, observe that two central irreducibles in $R$ are associates in $R$ if and only if they are associates in $R[y;\sigma,\delta]$. We conclude that, if $R$ is a commutative noetherian domain, $\sigma$ is an automorphism of $R$ and $\delta$ is a $\sigma$-derivation of $R$, $\sigma\neq id$ or $\delta\neq 0$, and $R$ admits sets of $\kappa$ non-associate central irreducibles, fixed by $\sigma$, and constant with respect to $\delta$, then $R[y;\sigma,\delta]$ is a non-commutative noetherian domain, admitting sets of $\kappa$ non-associate central irreducibles. The following are explicit examples of non-commutative noetherian domains, satisfying the conditions of Theorem 1.1, given by the above construction.

\begin{enumerate}
\item Let $S$ be a commutative noetherian domain. Let $\Omega$ be a finite nonempty set, algebraically independent over $S$. Let $\omega\in\Omega$. Let $R$ be the ring of ploynomials $S[\Omega]$, let $\sigma$ be equal to $0$ and let $\delta$ be equal to the formal partial derivative $\partial/\partial\omega$. Suppose $S$ admits sets of $\kappa$ non-associate irreducibles. If in this case we make $D$ to be equal to $R[y;\delta]$, then $D$ is a non-commutative noetherian domain admitting sets of $\kappa$ non-associate central irreducibles. In particular, if we make $S$ to be equal to $\mathbb{Z}[\sqrt{a}][\Omega_1][[\Omega_2]]$ as in example 1 above, by Theorem 1.1 we have

\[\chi(D)\geq 2^{\aleph_0}\]

\noindent and this bound is attained at least $2^{\aleph_0}$ times up to isomorphisms. Moreover, if we make $S$ to be equal to $k[\Omega_1][[\Omega_2]]$ as in example 2 above, again, by Theorem 1.1 we have

\[\chi(D)\geq 2^{\aleph_0\left|k\right|}\]

\noindent and this bound is attained at least $2^{\aleph_0\left|k\right|}$ times up to isomorphisms. Finally, if we make $S$ to be equal to $R_{(a)}$ as in example 3 above, again, by Theorem 1.1 we have

\[\chi(D)\geq 2^{\left\lfloor n/2\right\rfloor}-1\] 

\item Let $k$ be a field. Let $\sigma$ be an automorphism of $k$ such that $\sigma\neq id$. Let $\Omega$ be a finite set, algebraically independent over $k$. Let $R$ be equal to $k[\Omega]$. Let $\tilde{\sigma}$ be the only extension of $\sigma$ to $R$ such that $\tilde{\sigma}(\omega)=\omega$ for every $\omega\in\Omega$. Suppose $\sigma$ fixes nonzero elements of $k$. If in this case we make $D$ to be equal to $R[y;\tilde{\sigma}]$, then $D$ is a non-commutative noetherian domain admitting sets of $\aleph_0\left|k\right|$ non-associate central irreducibles. In particular, if we make $k$ to be equal to the algebraic closure, $\overline{\mathbb{F}}_p$, of $\mathbb{F}_p$ for some prime $p$, and we make $\sigma$ to be equal to any power of the Frobenius automorphism of $k$, then, by Theorem 1.1 we have

\[\chi(D)\geq 2^{\aleph_0}\]

\noindent and  this bound is attained at least $2^{\aleph_0}$ times up to isomorphisms. Moreover, if we now make $k$ to be equal to $\mathbb{C}$ and $\sigma$ to be equal to the complex conjugation in $\mathbb{C}$, then, again by Theorem 1.1 we have

\[\chi(D)\geq 2^\mathfrak{c}\]

\noindent and this bound is attained at least $2^\mathfrak{c}$ times up to isomorphisms.
\end{enumerate}

\noindent We end this section with the following problem related to examples 1 and 2 presented above.

\begin{prob}
Compute lower bounds for the following cardinals.
\begin{enumerate}
\item $\chi(k[y;\delta])$, where $(k,\delta)$ is any differential field.
\item $\chi(k[x][y;d])$, where $k$ is any field and $d$ denotes the formal derivative, with respect to variable $x$, in $k[x]$.
\item $\chi(\overline{\mathbb{F}}_p[y;\tilde{\sigma}])$, where $\sigma$ denotes the Frobenius automorphism in $\overline{\mathbb{F}}_p$
\item $\chi(\mathbb{C}[y,\tilde{\sigma}])$, where $\sigma$ denotes the complex conjugation on $\mathbb{C}$.
\end{enumerate}
\end{prob}

\section{Proof of Theorem 1.1}

\noindent In this final section we reduce the proof of Theorem 1.1 to the proof of a couple of lemmas. Observe first that the rank, $rk(M)$, of a module $M\in\mathfrak{S}$, is characterized by equation 

\[dim_EE\otimes_D M=rk(M)\]

\noindent or equivalently, by the fact that $M$ can be essentially embedded in $E^{(rk(M))}.$

\begin{lem}
Let $\left\{M_\lambda:\lambda\in\Lambda\right\}$ be a collection of modules. Let $\lambda_0\in \Lambda$ and for each $\lambda\in\Lambda$ such that $\lambda\neq \lambda_0$, let $a_\lambda\in M_\lambda$, $b_\lambda\in M_{\lambda_0}$, and $\xi_\lambda\in D\setminus \left\{0\right\}$. Suppose 

\begin{enumerate}
\item $M_\lambda$ is indecomposable for all $\lambda\in\Lambda$.
\item $Hom(M_\lambda ,M_\mu)=\left\{0\right\}$ for all $\lambda, \mu\in\Lambda$ with $\lambda\neq\mu$.
\item $a_\lambda$ is not divisible by $\xi_\lambda$ in $M_\lambda$ for each $\lambda\in\Lambda\setminus\left\{\lambda_0\right\}$.
\item $b_\lambda$ is not divisible by $\xi_\lambda$ in $M_{\lambda_0}$ for each $\lambda\in\Lambda\setminus\left\{\lambda_0\right\}$.
\end{enumerate}
 
\noindent Then the $D$-submodule

\[M=\bigoplus_{\lambda\in\Lambda}M_\lambda +\sum_{\lambda\in\Lambda\setminus\left\{\lambda_0\right\}}D\xi_\lambda^{-1}(a_\lambda+b_\lambda)\]

\noindent of $\bigoplus_{\lambda\in\Lambda}(E\otimes_D M_\lambda)$ is indecomposable. Moreover, if $M_\lambda\in\mathfrak{S}$ for each $\lambda\in \Lambda$, then

\[rk(M)=\left|\Lambda\right|\sum_{\lambda\in \Lambda}rk(M_\lambda)\]
\end{lem}

\begin{proof}

Let $M=N\oplus L$ be a direct sum decomposition of $M$. By condition 2 above, $M_\lambda$ is strongly invariant in $M$ for each $\lambda\in \Lambda$. The identity

\[M_\lambda=(M_\lambda\cap L)\oplus (M_\lambda\cap N)\]

\noindent follows for all $\lambda\in \Lambda$. Thus, since by condition 1 above $M_\lambda$ is indecomposable for each $\lambda\in\Lambda$, each $M_\lambda$ is contained in either $N$ or $L$. Suppose $M_{\lambda_0}\leq N$. Suppose now that there exists $\lambda\in\Lambda$ such that $M_\lambda\leq L$. In this case equation 

\[\xi_\lambda^{-1}(a_{\lambda}+b_\lambda)=x+y\]

\noindent with $x\in N$ and $y\in L$, implies that $\xi_\lambda$ divides $a_\lambda$ in $L$ and $b_\lambda$ in $N$, which in turn implies that $\xi_\lambda$ divides both $a_\lambda$ and $b_\lambda$ in $M$, a contradiction, since by conditions 3 and 4 above, niether $a_\lambda$ nor $b_\lambda$ are divisible by $\xi_\lambda$ in $\bigoplus_{\lambda\in\Lambda}M_\lambda$, and from this and from the definition of $M$ it follows that niether $a_\lambda$ nor $b_\lambda$ are divisible by $\xi_\lambda$ in $M$. It follows that $M_\lambda\leq N$ for all $\lambda\in \Lambda$. Now, since $\bigoplus_{\lambda\in\Lambda}M_\lambda$ is clearly essential in $M$ we conclude that $M$ is indecomposable. The conclusion on the rank of $M$ is trivial.
\end{proof}

\begin{lem}
Let $\kappa_1,\kappa_2$ be cardinals. Suppose $D$ admits sets of $2\kappa_1+\kappa_2$ non-associate central irreducibles and that $D$ does not admit non-zero infinitely left divisible elements. Then, in this case, there exists a collection of modules $\left\{M_\lambda:\lambda\in\Lambda\right\}$, of cardinality $\kappa_1$, an element $\lambda_0\in\Lambda$, and elements $a_\lambda\in M_\lambda$, $b_\lambda\in M_{\lambda_0}$, and $\xi_\lambda\in D\setminus \left\{0\right\}$, $\lambda\in\Lambda\setminus\left\{\lambda_0\right\}$, all satisfying the conditions of Lemma 3.1.
\end{lem}

\begin{proof} 
Let $\left\{a_\gamma:\gamma\in \Gamma \right\}, \left\{b_\gamma:\gamma \in\Gamma\right\}$, $\left\{c_\delta :\delta\in\Delta\right\}$ be three disjoint sets consisting of different central irreducibles in $D$ such that no two elements in the union of the three sets are associates in $D$. Suppose $\left|\Gamma\right|=\kappa_1$ and $\left|\Delta\right|=\kappa_2$. Let $\Theta$ be the collection of all sets of the form 

\[\left\{d_\gamma: d_\gamma\in \left\{a_\gamma ,b_\gamma\right\} \forall\gamma\in\Gamma\right\}\]

\noindent It is immediate that $\left|\Theta\right|=2^{\kappa_1}$. Now, for each $A\in \Theta$, let $M_A$ be the following submodule of $E$

\[\sum_{d_{\gamma}\in A}D[1/d_\gamma]\]

\noindent For each $A\in \Theta$, $M_A$ thus defined, being a submodule of the uniform module $E$, is itself uniform, and thus indecomposable, that is, the collection $\left\{M_A:A\in \Theta\right\}$ satisfies condition 1 of Lemma 3.1. Now, let $A,B\in \Theta$ such that $A\neq B$, we prove that $Hom(M_A ,M_B)=\left\{0\right\}$. From the definition of $\Theta$ it is clear that in this case niether $A$ is contained in $B$ nor $B$ is contained in $A$. Let $d_{\gamma}\in A\setminus B$. Let $\Phi\in Hom(M_A ,M_B)$. From the fact that for each $\gamma\in\Gamma$, $a_\gamma$ and $b_\gamma$ are non-associate, central, and irreducible, it is easily seen that no non-zero element in $M_B$ is infinitely divisible by $d_{\gamma}$. From this and from the fact that $1_D$ is infinitely divisible by $d_\gamma$ in $M_A$ it follows that $\Phi(1_D)=0$, and since $D$ is essential in $M_A$, we conclude, from this, that $\Phi=0$. Thus the collection $\left\{M_A:A\in \Theta\right\}$ satisfies condition 2 of Lemma 3.1. Finally, let $\Theta$ play the rôle of $\Lambda$, let $A_0\in \Theta$ play the rôle of $\lambda_0$, and let any sequence formed by elements in $\left\{c_\delta :\delta\in\Delta\right\}$ play the rôle of the sequence $\xi_\lambda$, $\lambda\in\Lambda\setminus\left\{\lambda_0\right\}$ in Lemma 3.1. Finally, let any sequence of elements in $A$ and in any sequence of elements in $A_0$, play the rôle of elements $a_\lambda$ and $b_\lambda$ in Lemma 3.1. It is immediate that in this case collection $\left\{M_A :A\in \Theta\right\}$, together with $A_0$ and corresponding subsets of $\left\{c_\delta :\delta\in\Delta\right\}$, $A$, and $A_0$ satisfy the conditions of Lemma 3.1. This concludes the proof.
\end{proof}

\noindent With the aid of Lemma 3.1 and Lemma 3.2 we now prove Theorem 1.1.

\

\noindent \textbf{\textit{Proof of theorem 1.1}} Let $\left\{a_\gamma:\gamma\in\Gamma\right\}$, $\left\{b_\gamma:\gamma\in\Gamma\right\}$, $\left\{c_\delta:\delta\in\Delta\right\}$, and $\Theta$ be as in the proof of lemma 3.2. It follows, from Lemma 3.1 and from the proof of Lemma 3.2 that every sequence with indeces in $\Theta$, formed by elements in any non-empty subset of $\left\{c_\delta:\delta\in\Delta\right\}$, defines an indecomposable module of rank $2^{\kappa_1}$. Let $S,T$ be two different non-empty subsets of $\left\{c_\delta:\delta\in\Delta\right\}$. Let $s,t$ be two sequences with indeces in $\Theta$, formed by elements in $S$ and $T$ respectively. Finally, let $M_s,M_t$ be indecomposable modules of rank $2^{\kappa_1}$, defined, as in the proof of Lemma 3.2, by sequences $s$ and $t$ respectively. Suppose, without any loss of generality, that $S\setminus T\neq \emptyset$. Let $a\in S\setminus T$. Then, by the way they were constructed, $M_s$ admits non-zero elements divisible by $a$, while $M_t$ does not. It follows that $M_s$ and $M_t$ are non-ismorphic. We conclude, from this, from Lemma 3.1, and from Lemma 3.2, that there exist at least $2^{\kappa_2}$ non-ismorphic indecomposable modules of rank $2^{\kappa_1}$. This concludes the proof. $\blacksquare$

\end{document}